\documentclass{amsart}
\usepackage{amssymb}
\usepackage{amsfonts}
\usepackage{amssymb}
\usepackage{amsmath}
\usepackage{amsthm}
\usepackage{enumerate}
\usepackage{tabularx}
\usepackage{centernot}
\usepackage{mathtools}
\usepackage{stmaryrd}
\usepackage{amsthm,amssymb}
\usepackage{amssymb}
\usepackage{amsfonts}
\usepackage{amssymb}
\usepackage{amsmath}
\usepackage{amsthm}
\usepackage{enumerate}
\usepackage{tabularx}
\usepackage{url}
\usepackage{centernot}
\usepackage{mathtools}
\usepackage{stmaryrd}
\usepackage{amsthm,amssymb}
\usepackage{etoolbox}
\usepackage{tikz}
\usepackage{amssymb}
\usetikzlibrary{matrix}
\usepackage{tikz-cd}
\usepackage{tikz}
\usepackage{marginnote}
\definecolor{mygray}{gray}{0.85}
\usepackage[backgroundcolor=mygray,colorinlistoftodos,prependcaption,textsize=small]{todonotes}
\usepackage{xargs}                      

\renewcommand{\leq}{\leqslant}

\makeatletter
\def\subsection{\@startsection{subsection}{3}%
  \z@{.5\linespacing\@plus.7\linespacing}{.3\linespacing}%
  {\bfseries\centering}}
\makeatother

\makeatletter
\def\subsubsection{\@startsection{subsubsection}{3}%
  \z@{.5\linespacing\@plus.7\linespacing}{.3\linespacing}%
  {\centering}}
\makeatother

\makeatletter
\def\myfnt{\ifx\protect\@typeset@protect\expandafter\footnote\else\expandafter\@gobble\fi}
\makeatother

\newtheorem{theorem}{Theorem}

\newtheorem{corollary}[theorem]{Corollary}
\newtheorem{definition}[theorem]{Definition}
\newtheorem{lemma}[theorem]{Lemma}

\newtheorem{problem}[theorem]{Problem}
\newtheorem{convention}[theorem]{Convention}
\newtheorem{question}[theorem]{Question}

\newtheorem{remark}[theorem]{Remark}

\newcounter{claimcounter}
\numberwithin{claimcounter}{theorem}

\newcommand{\pureindep}[1][]{%
  \mathrel{
    \mathop{
      \vcenter{
        \hbox{\oalign{\noalign{\kern-.3ex}\hfil$\vert$\hfil\cr
              \noalign{\kern-.7ex}
              $\smile$\cr\noalign{\kern-.3ex}}}
      }
    }\displaylimits_{#1}
  }
}

\begin{document}

\begin{abstract} We construct a $\wedge$-homogeneous universal simple matroid of rank $3$, i.e. a countable simple rank~$3$ matroid $M_*$ which $\wedge$-embeds every finite simple rank $3$ matroid, and such that every isomorphism between finite $\wedge$-subgeometries of $M_*$ extends to an automorphism of $M_*$. We also construct a $\wedge$-homogeneous matroid $M_*(P)$ which is universal for the class of finite simple rank $3$ matroids omitting a given finite projective plane $P$.  We then prove that these structures are not $\aleph_0$-categorical, they have the independence property, they admit a stationary independence relation, and that their automorphism group embeds the symmetric group $Sym(\omega)$. Finally, we use the free projective extension $F(M_*)$ of $M_*$ to conclude the existence of a countable projective plane embedding all the finite simple matroids of rank $3$ and whose automorphism group contains $Sym(\omega)$, in fact we show that $Aut(F(M_*)) \cong Aut(M_*)$.
\end{abstract}

\title{A Universal Homogeneous Simple Matroid of Rank~$3$}
\thanks{Partially supported by European Research Council grant 338821.}

\author{Gianluca Paolini}
\address{Einstein Institute of Mathematics,  The Hebrew University of Jerusalem, Israel}

\date{\today}
\maketitle


\section{Introduction}

	A countably infinite structure $M$ is said to be {\em homogeneous} if every isomorphism between finitely generated substructures of $M$	extends to an automorphism of $M$. The study of homogeneous combinatorial structures such as graph, digraphs and hypergraphs is a very rich field of study (see e.g. \cite{cherlin1}, \cite{cherlin2}, \cite{jb1}, \cite{jb2}, \cite{lach1}, \cite{lach2}). On the other hand, {\em matroids} are objects of fundamental importance in combinatorial theory, but very little is known on homogeneous matroids. In this short note we propose a new approach to the study of homogeneous matroids, focusing on the case in which the matroid is of rank $3$ and simple. In this case the matroidal structure can be defined in a very simple manner as a $3$-hypergraph\footnote{To the reader familiar with matroid theory it will be clear that in $(V, R)$ the hyper-edges are nothing but the dependent sets of size $3$ of the unique simple matroid of rank $3$ coded by $(V, R)$.}, as follows:
	
	\begin{definition}\label{def_matroid} A {\em simple matroid\footnote{For a general introduction to matroid theory see e.g. the classical references \cite{rota} and \cite{kung}.} of rank $\leq 3$} is a $3$-hypergraph $(V, R)$ whose adjacency relation is irreflexive, symmetric \mbox{and satisfies the following {\em exchange axiom}:}
	\begin{enumerate}[(Ax)]
	\item if $R(a, b, c)$ and $R(a, b, d)$, then $\{a, b, c, d \}$ is an $R$-clique.
\end{enumerate}
We say that the matroid has rank $3$ if it contains three non-adjacent points.
\end{definition} 
	
	As well-known (see e.g. \cite[pg. 148]{kung}), simple matroids of rank $ \leq 3$ are in canonical correspondence (cf. Convention \ref{sloppy}) with certain incidence structures known as {\em linear spaces}:
	
	\begin{definition}\label{def_plane} A {\em linear space} is a system of points and lines satisfying:
	\begin{enumerate}[(A)]
	\item every pair of distinct points determines a unique line;
	\item every pair of distinct lines intersects in at most one point;
	\item every line contains at least two points.
\end{enumerate}
\end{definition}

	In \cite{devillers} Devellers provides a complete classification of the countable homogeneous linear spaces. In this work it is shown that (as formulated) the theory is very poor, and in fact the only infinite homogeneous linear space is the trivial one, i.e. infinitely many points and infinitely many lines incident with exactly two points.
	
	This situation is reflected in the context of matroid theory with the well-known observation (see e.g. \cite[Example 7.2.3]{oxley}) that the class of finite simple matroids of rank $3$ does not have the amalgamation property, and so the construction of a homogeneous (with respect to the notion of subgeometry) simple matroid of rank $3$ containing all the finite simple matroids of rank $3$ as subgeometries is hopeless.  
	
	One might wonder if this is all there is to it, and no further mathematical theory is possible. In this short note we give evidence that this is not the case, and that there might be a very interesting combinatorial theory for homogeneous matroids, if the problem (viz. choice of language) is correctly formulated.
	
	The crucial observation that underlies our approach is that (with respect to questions of homogeneity) the choice of substructure that we are considering is {\em too weak}, and does not take into account enough of the geometric structure encoded by these objects, i.e. their associated {\em geometric lattices}\footnote{A geometric lattice is a semi-modular point lattice without infinite chains. For more on this see e.g. \cite[Section 2]{paolini&hyttinen}, \cite[Chapter 2]{rota} and \cite{kung}.}. This inspires:
	
	\begin{definition}\label{def_geom_lat} Let $P$ be a linear space (cf. Definition \ref{def_plane}). On $P$ we define two partial functions $p_1 \vee p_2$ and $\ell_1 \wedge \ell_2$ denoting, respectively, the unique line passing through the points $p_1$ and $p_2$, and the unique point $p$ at the intersection of the lines $\ell_1$ and $\ell_2$, if such a point exists, and $0$ otherwise (where $0$ is a new symbol). If we extend $P$ to $\hat{P}$ adding a largest element $1$ and a smallest element $0$ and we extend the interpretation of $\vee$ and $\wedge$ in the obvious way, then the structure $(\hat{P}, \vee, \wedge, 0, 1)$ is a so-called geometric lattice. For details on this \mbox{see \cite[Chapter 2]{rota} or \cite[pg. 148]{kung}.}
\end{definition}

	\begin{convention}\label{sloppy} When convenient, we will be sloppy in distinguishing between a simple rank~$3$ matroid and it associated linear space/geometric lattice (cf. Definition~\ref{def_geom_lat}). This is justified by the following canonical correspondence between the two classes of structures. Given a linear space $P$ consider the simple rank $3$ matroid $M_P$ whose dependent sets of size $3$ are the triples of collinear points of $P$. Given a simple rank~$3$ matroid $(V, R)$ consider the linear space $P_M$ whose points $p$ are the elements of $V$ and whose lines $\ell$ are the sets of the form $\{a, b \} \cup \{ c \in V : R(a, b, c) \}$, together with the incidence relation  $p \in \ell$. Also, we will use freely the partial functions $\vee$ and $\wedge$ introduced in Definition \ref{def_geom_lat} in the context of linear spaces.
\end{convention}
	
	\begin{definition}\label{def_wedge_matroid} A {\em simple $\wedge$-matroid of rank $ \leq 3$} is a structure $M = (V, R, \wedge)$ such that $(V, R)$ is a simple matroid of rank $ \leq 3$ (cf. Definition \ref{def_matroid}) and $\wedge$ is a $4$-ary function defined as follows\footnote{Clearly, in the definition of $\wedge(a, b, c, d)$, the only case in which we are interested is the first case of the disjunction, i.e. when $a \vee b$ and $c \vee d$ are two distinct lines intersecting in a fifth point $p$, in which case the vale of $\wedge(a, b, c, d)$ is indeed $p$. The way the definition of $\wedge(a, b, c, d)$ is written is just a technical way to express this natural condition.} (cf. Definition \ref{def_geom_lat}):
	$$\wedge_M(a, b, c, d) = \begin{cases} (a \vee b) \wedge (c \vee d) \;\; \text{ if } (a \vee b) \wedge (c \vee d) \notin \{ 0 , a, b, c, d, a \vee b, c \vee d \},\\ a \;\;\;\;\;\;\;\;\;\;\;\;\;\;\;\;\;\;\;\;\;\;\;\;\;\; \text{otherwise}.
	\end{cases}$$
\end{definition} 

	In this study we will see that with respect to the new notion of substructure introduced in Definition \ref{def_wedge_matroid} there is hope for a rich mathematical theory, which is potentially analogous to the situation for homogeneous graphs (see e.g. \cite{lach1}). In fact, we prove:

	\begin{theorem}\label{th_construction} There exists a homogeneous simple rank $3$ $\wedge$-matroid $M_*$ which is universal for the class of finite simple $\wedge$-matroids of rank $\leq 3$.
\end{theorem}

\begin{theorem}\label{th_construction_omitting} Let $P$ be a finite projective plane, and $M_P$ the corresponding simple rank~$3$ matroid. Then there exists a homogeneous simple $\wedge$-matroid $M_*(P)$ which is universal for the class of finite simple $\wedge$-matroids of rank $\leq 3$ omitting\footnote{By this we mean that there is no injective map $f: M_P \rightarrow N$ such that $M_P \cong f(M_P)$.} $M_P$.
\end{theorem}

	It might be argued that in the context of simple rank $3$ matroids the homogeneous structure of Theorem \ref{th_construction} plays the role played by the random graph \cite{rado} for the class of finite graphs, while the homogeneous structure of Theorem \ref{th_construction} plays the role played by the random $K_n$-free\footnote{$K_n$ denotes the complete graph on $n$ vertices.} graph \cite{henson_graphs} for the class of finite graph omitting $K_n$.
	
		We then prove several facts of interest on the automorphism groups of the homogeneous structures from Theorems \ref{th_construction} and \ref{th_construction_omitting}.

	\begin{theorem}\label{th_properties} Let $M_*$ be as in Theorem \ref{th_construction} or Theorem \ref{th_construction_omitting}. Then:
	\begin{enumerate}[(1)]
	\item $M_*$ is not $\aleph_0$-categorical;
	\item $M_*$ has the independence property;
	\item $M_*$ admits a stationary independence relation;
    \item $Aut(M_*)$ embeds the symmetric group $Sym(\omega)$;
    \item if the age of $M_*$ has the extension property for partial automorphisms, then $Aut(M_*)$ has ample generics, and in particular it has the small index property.
\end{enumerate}	
\end{theorem}

	Finally, we give an application to projective geometry proving:

	\begin{corollary}\label{proj_corollary} Let $M_*$ be as in Theorem \ref{th_construction}, and let $F(M_*)$ be the free projective extension of $M_*$ (cf. \cite{hall_proj}). Then:
	\begin{enumerate}[(1)]
	\item $F(M_*)$ embeds all the finite simple rank $3$ matroids as subgeometries;
	\item every $f \in Aut(M_*)$ extends to an $\hat{f} \in Aut(F(M_*))$;
	\item $f \mapsto \hat{f}$ is an isomorphism from $Aut(M_*)$ onto $Aut(F(M_*))$;
 	\item $Aut(F(M_*))$ embeds the symmetric group $Sym(\omega)$.
\end{enumerate}
\end{corollary}	

	We leave the following open questions:
	
	\begin{question}\label{question} Let $M_*$ be as in Theorem \ref{th_construction} or Theorem \ref{th_construction_omitting}.
	\begin{enumerate}[(1)]
	\item Does $Aut(M_*)$ have the small index property?
	\item Does $Aut(M_*)$ have ample generics?
\end{enumerate}
\end{question}

	\begin{question} \begin{enumerate}[(1)]
	\item Does the class of simple $\wedge$-matroids of rank $3$ have the extension property for partial automorphisms?
	\item Does the class of freely linearly ordered simple $\wedge$-matroids of rank $3$ have the Ramsey property?
\end{enumerate}	
\end{question}

	The only infinite homogeneous simple $\wedge$-matroids of rank $3$ known to the author are the ones from Theorems \ref{th_construction} and \ref{th_construction_omitting}, and the trivial one, i.e.  infinitely many points and infinitely many lines incident with exactly two points.
	
\begin{problem}\label{question_homogeneous} Classify the countable  homogeneous simple $\wedge$-matroids of rank $3$.
\end{problem}
		
	Concerning $F(M_*)$, in \cite{proj_embeddings} Kalhoff constructs a projective plane of Lenz-Barlotti class V embedding all the finite simple rank $3$ matroids. In \cite{baldwin} Baldwin constructs some almost strongly minimal projective planes of Lenz-Barlotti class I.1. We leave as an open problem the determination of the Lenz-Barlotti class of $F(M_*)$.  
	

\section{Preliminaries}

	For background on Fra\"iss\'e theory and homogeneous structures we refer to \cite[Chapter 6]{hodges}. In particular, given a homogeneous structure $M$ we refer to the closure up to isomorphisms of the collection of finitely generated substructures of $M$ as the {\em age} of $M$ and denote it by $\mathbf{K}(M)$. For background on the notions on automorphism groups occurring in Theorem \ref{th_properties} see e.g. \cite{kechris}. Concerning free projective extensions see \cite{hall_proj} and \cite[Chapter XI]{piper}. Concerning the notion of stationary independence:
	
		\begin{definition}[\cite{tent_ziegler} and \cite{muller}] Let $M$ be a homogeneous structure. We say that a ternary relation $A \pureindep[C] B$ between finitely generated substructures of $M$ is a {\em stationary independence relation} if the following axioms are satisfied:
\begin{enumerate}[$(A)$]
		\item (Invariance) if $A \pureindep[C] B$ and $f \in Aut(M)$, then $f(A) \pureindep[f(C)] f(B)$;
		\item (Symmetry) if $A \pureindep[C] B$, then $B \pureindep[C] A$;
		\item (Monotonicity) if $A \pureindep[C] \langle BD \rangle$ and $A \pureindep[C] B$, then $A \pureindep[\langle BC \rangle] D$;
		\item (Existence) there exists $A' \equiv_B A$ such that $A' \pureindep[B] C$;
		\item (Stationarity) if $A \equiv_C A'$, $A \pureindep[C] B$ and $A' \pureindep[C] B$, then $A \equiv_{\langle BC \rangle} A'$.
\end{enumerate}
\end{definition}

	\begin{definition} A {\em projective plane} is a linear space (cf. Definition \ref{def_plane}) such that:
	\begin{enumerate}[(A')]
	\item every pair of distinct lines intersects in a unique point;
	\item there exist at least four points no three of which are collinear.
\end{enumerate}
\end{definition}

	For a definition of the notion of independence property of a first-order theory see e.g. \cite[Exercise 8.2.2]{ziegler_book}.

\section{Proofs}

	We will prove a series of claims from which Theorems \ref{th_construction}, \ref{th_construction_omitting} and \ref{th_properties} follow.
	
	\begin{lemma}\label{theorem_fraisse} The class of simple $\wedge$-matroids of rank $ \leq 3$ is a Fra\"iss\'e class.
\end{lemma}

	\begin{proof} The hereditary property is clear. The joint embedding property is easy and the amalgamation property is proved in \cite[Theorem 4.2]{paolini&hyttinen}. Notice that the context of \cite{paolini&hyttinen} is the study of geometric lattices in a language $L' = \{ 0, 1, \vee, \wedge \}$, but keeping in mind Definition \ref{def_geom_lat}, Convention \ref{sloppy}, and the fact that we are considering $\wedge$-matroids it is easy to see that the two context are indeed equivalent.
\end{proof}

\begin{definition}[{\cite[Definition 6]{ssip_canonical_hom}}]\label{can_amal} Let $M$ be a homogeneous structure and $\mathbf{K} = \mathbf{K}(M)$ its age. We say that $M$ has {\em canonical amalgamation} if there exists an operation $B_1 \oplus_A B_2$ on triples from $\mathbf{K}$ satisfying the following conditions:
	\begin{enumerate}[(a)]
	\item $B_1 \oplus_A B_2$ is defined when ${A} \subseteq {B}_i$ ($i = 1,2$) and $B_1 \cap B_2 = A$;
	\item $B_1 \oplus_A B_2$ is an amalgam of $B_1$ and $B_2$ over $A$;
	\item if $B_1 \oplus_A B_2$ and $B'_1 \oplus_{A'} B'_2$ are defined, and there exist $f_i: B_i \cong B'_i$ ($i = 1,2$) with $f_1 \restriction A = f_2 \restriction A$, then there is:
	$$f: B_1 \oplus_A B_2 \cong B'_1 \oplus_{A'} B'_2$$
such that $f \restriction B_1 = f_1$ and $f \restriction B_2 = f_2$.
\end{enumerate}
\end{definition}

	\begin{remark}\label{remark_canonical_amalgam} Notice that the amalgamation from \cite[Theorem 4.2]{paolini&hyttinen} used to prove Lemma \ref{theorem_fraisse} is canonical in the sense of Definition \ref{can_amal}. We will denote the canonical amalgam of $M_1$ and $M_2$ over $M_0$ from \cite[Theorem 4.2]{paolini&hyttinen} as $M_1 \oplus_{M_0} M_2$ (when we use this notation we tacitly assume that $M_0 \subseteq M_1$, $M_0 \subseteq M_2$ and $M_1 \cap M_2 = M_0$). Notice that the amalgam $M_3 := M_1 \oplus_{M_0} M_2$ can be characterized as the following $\wedge$-matroid:
	\begin{enumerate}[(1)]
	\item $M_3 = M_1 \cup M_2$ (i.e. $M_1 \cup M_2$ is the domain of $M_3$);
	\item $R^{M_3} = R^{M_1} \cup R^{M_2} \cup \{ \{ a, b, c \} : a \vee b = a \vee c = b \vee c = a' \vee b' \text{ and } \{ a', b' \} \subseteq M_0 \}$;
	\item\label{char_item} $\wedge_{M_3}(a, b, c, d) = a$, unless $a \vee b = a' \vee b'$, $c \vee d = c' \vee d'$ and $\wedge_{M_\ell}(a', b', c', d') \neq a'$, for some $\ell = 1, 2$ and $\{ a', b', c', d' \} \subseteq M_\ell$, in which case:
	$$\wedge_{M_3}(a, b, c, d) = \wedge_{M_\ell}(a', b', c', d').$$
\end{enumerate}
The intuition behind (\ref{char_item}) is that the value of the function symbol $\wedge_{M_3}(a, b, c, d)$ is trivial unless $a \vee b$ and $c \vee d$ are two intersecting lines from one of the $M_\ell$ ($\ell = 1, 2$).
\end{remark}

	\begin{lemma}\label{theorem_fraisse_omitting} Let $P$ be a finite projective plane, and $M_P$ the corresponding matroid. The class of simple $\wedge$-matroids $N$ of rank $\leq 3$ omitting\footnote{Recall that by this we mean that there is no injective map $f: M_P \rightarrow N$ such that $M_P \cong f(M_P)$.} $M_P$ is a Fra\"iss\'e class.
\end{lemma}

	\begin{proof} Also in this case, the only non-trivial part of the proof is amalgamation. Let $M_0, M_1, M_2$ be $\wedge$-matroids omitting $M_P$ and such that $M_0 \subseteq M_1$, $M_0 \subseteq M_2$ and $M_1 \cap M_2 = M_0$. Let $M_3 := M_1 \oplus_{M_0} M_2$ be as in Remark \ref{remark_canonical_amalgam}. We want to show that $M_3$ does not embed $M_P$, but this is clear noticing that by Remark \ref{remark_canonical_amalgam} we have: 
	\begin{enumerate}[(i)]
	\item if $j \in \{ 1, 2 \}$ and $\ell$ is a line from $M_j$ such that there are no $a_0, a_1$ in $M_0$ with $\ell = a_0 \vee a_1$, then the number of points incident with $\ell$ in $M_j$ is equal to the number of points incident with $\ell$ in $M_3$;
	\item if $\ell$ is a line of $M_3$ which is incident with at most one point of $M_1$ and  at most one point of $M_2$, respectively, then $\ell$ is incident with exactly two points.
\end{enumerate}	 
\end{proof}

	We can now prove Theorems \ref{th_construction} and \ref{th_construction_omitting}.

	\begin{proof}[Proof of Theorems \ref{th_construction} and \ref{th_construction_omitting}] This follows from Lemmas \ref{theorem_fraisse} and \ref{theorem_fraisse_omitting} using Fra\"iss\'e theory (see e.g. \cite[Chapter 6]{hodges}).
\end{proof}

	The following lemma establishes the non $\aleph_0$-categoricity of the homogeneous structures of Theorems \ref{th_construction} and \ref{th_construction_omitting}.

	\begin{lemma}\label{lemma_for_oligo} For every $n < \omega$ there exists a finite simple rank $3$ $\wedge$-matroid $M(n)$ of size $6 + (n+1)$, and $6$ distinct points $p_1, .., p_6 \in M(n)$ such that $\langle p_1, ..., p_6 \rangle_{M(n)} = M(n)$, where $\langle A \rangle_B$ denotes the substructure generated by $A$ in $B$. Furthermore, $M(n)$ can be taken such that it does not contain any projective plane.
\end{lemma}

	\begin{proof} By induction on $n < \omega$, we construct a finite simple rank $3$ $\wedge$-matroid $M(n)$ such that:
	\begin{enumerate}[(a)]
	\item the domain of $M(n)$ is $\{ p_1^{-}, p_2^{-}, p_1^{+}, p_2^{+}, p_1^{*}, p_2^{*}, q_0, ..., q_n \}$;
	\item $|\{ p_1^{-}, p_2^{-}, p_1^{+}, p_2^{+}, p_1^{*}, p_2^{*}, q_0, ..., q_n \}| = 6 + (n+1)$;
	\item if $n$ is even, then $p_1^{-} \vee q_n$ and $p_2^{-} \vee p^*_1$ are parallel in $M(n)$;
	\item if $n$ is odd, then $p_1^{+} \vee q_n$ and $p_2^{+} \vee p^*_2$ are parallel in $M(n)$;
	\item $\langle p_1^{-}, p_2^{-}, p_1^{+}, p_2^{+}, p^*_1, p^*_{2} \rangle_{M(n)} = M(n)$.
\end{enumerate}
\underline{$n = 0$}. Let $M$ be the simple rank $3$ $\wedge$-matroid with domain $\{ p_1^{-}, p_2^{-}, p_1^{+} p_2^{+}, p_1^{*}, p_2^{*} \}$ such that are no hyper-edges (i.e. every line is incident with exactly two points). Add to $M$ the point $q_0$ which is incident only with the lines $p_1^{+} \vee p_2^{+}$ and $p_1^{*} \vee p_2^{*}$ (which are parallel in $M$), and let $M(0)$ be the resulting $\wedge$-matroid.
\newline \underline{$n = 2k+1$}. Let $M(n-1)$ be constructed, then $M(n-1)$ contains the lines $p_1^{-} \vee q_{2k}$ and $p_2^{-} \vee p^*_1$, and, by induction hypothesis, they are parallel in $M(n-1)$. Add to $M(n-1)$ the point $q_n$ which is incident only with the lines $p_1^{-} \vee q_{2k}$ and $p_2^{-} \vee p^*_1$ which are parallel in $M(n-1)$), and let $M(n)$ be the resulting $\wedge$-matroid.
\newline \underline{$n = 2k > 0$}. Let $M(n-1)$ be constructed, then $M(n-1)$ contains the lines $p_1^{+} \vee q_{2k-1}$ and $p_2^{+} \vee p^*_{2}$, and, by induction hypothesis, they are parallel in $M(n-1)$. Add to $M(n-1)$ the point $q_n$ which is incident only with the lines $p_1^{+} \vee q_{2k-1}$ and $p_2^{+} \vee p^*_{2}$, and let $M(n)$ be the resulting $\wedge$-matroid.
\end{proof}

	\begin{lemma}\label{lemma_for_indep_prop} Let $M_*$ be as in Theorem \ref{th_construction} or Theorem \ref{th_construction_omitting}. Then $M_*$ has the independence property.
\end{lemma}

	\begin{proof} As in \cite[Theorem 4.6]{paolini&hyttinen}.
\end{proof}

	\begin{lemma}\label{lemma_for_stationary} Let $M_*$ be as in Theorem \ref{th_construction} or Theorem \ref{th_construction_omitting}. For every finite substructures $A, B, C$ of $M_*$, define $A \pureindep[C] B$ if and only if $\langle A, B, C \rangle_{M_*} \cong \langle A, C \rangle_{M_*} \oplus_C \langle B, C \rangle_{M_*}$. Then $A \pureindep[C] B$ is a stationary independence relation.
\end{lemma}

	\begin{proof} Easy to see using Remark \ref{remark_canonical_amalgam}.
\end{proof}

	\begin{lemma}\label{lemma_for_completeness} Let $M_*$ be as in Theorem \ref{th_construction} or Theorem \ref{th_construction_omitting}. If $f  \in Sym(M_*)$ induces an automorphism of $Aut(M_*)$ (i.e. $g \mapsto fgf^{-1} \in Aut(Aut(M_*))$), then $f \in Aut(M_*)$.
\end{lemma}

	\begin{proof} First of all, notice that if $M$ is a simple rank $3$ $\wedge$-matroid and $M^-$ is the reduct of $M$ to the language $L = \{ R \}$ then we have that $f \in Aut(M)$ if and only if $f \in Aut(M^-)$. Thus if $f \notin Aut(M_*)$, then $f \notin Aut(M^-_*)$, i.e. there exists a set $\{ a, b, c \} \subseteq M_*$ such that either $\{ a, b, c \}$ is dependent in $M_*$ and $\{ f(a), f(b), f(c) \}$ is independent in $M_*$, or $\{ a, b, c \}$ is independent in $M_*$ and $\{ f(a), f(b), f(c) \}$ is dependent in $M_*$. 
Modulo replacing $f$ with $f^{-1}$, we can assume, that $\{ a, b, c \}$ is independent in $M_*$ and $\{ f(a), f(b), f(c) \}$ is dependent in $M_*$. 
Suppose now that in addition $f$ induces an automorphism of $Aut(M_*)$. Since $M_*$ is homogeneous, it is easy to see that $f$ has to send dependent sets of size $3$ to independent sets of size $3$. Now, by Definition \ref{def_matroid}, if $R(a, b, c)$ and $R(a, b, d)$, then $\{a, b, c, d \}$ is an $R$-clique. On the other hand, trivially in $M_*$ we can find distinct points $\{ a, b, c, d \}$ such that $\{ a, b, c \}$ is an independent set, $\{ a, b, d \}$ is an independent set, and $\{b, c, d \}$ is {\em not} an independent set. Hence, we easily reach a contradiction.
\end{proof}

	\begin{remark}\label{lemma_for_Sym} Notice that the linear space $P$ consisting of infinitely many points and infinitely many lines incident with exactly two points satisfies $Aut(P) \cong Sym(\omega)$.
\end{remark}

	\begin{proof}[Proof of Theorem \ref{th_properties}] Item (1) follows from Lemma \ref{lemma_for_oligo}. Item (2) is Lemma \ref{lemma_for_indep_prop}. Item (3) follows from Lemma \ref{lemma_for_stationary}. Item (4) follows from item (3), the main result of \cite{muller} and Remark \ref{lemma_for_Sym}. Item (5) follows from Remark \ref{remark_canonical_amalgam} (JEP for partial automorphisms is easy to see), \cite[Theorem 1.6]{kechris}, and \cite[Theorem 6.2]{kechris}. 
\end{proof}

	\begin{proof}[Proof of Corollary \ref{proj_corollary}] Notice that every point and every line of $M_*$ is contained in a copy of the Fano plane (which is a confined configuration, in the terminology of \cite[pg. 220]{piper}). Thus, the result follows from \cite[Theorem 11.18]{piper} \mbox{or \cite[Lemma 1]{projective}.}
\end{proof}

\end{document}